\documentclass[11pt,reqno]{amsart}
\usepackage[utf8]{inputenc}
\usepackage[T1]{fontenc}
\usepackage{lmodern}
\usepackage{enumitem}
\usepackage{fullpage}
\usepackage{hyperref}
\usepackage{stmaryrd, amsthm, mathtools, amssymb, dsfont, mathrsfs, bm,amsmath,bbm}
\usepackage{booktabs,diagbox}
\usepackage[noabbrev,capitalise]{cleveref}
\usepackage[ruled,nosemicolon,vlined]{algorithm2e}
\usepackage{tikz}
\usetikzlibrary{decorations.pathreplacing,positioning}
\usepackage{todonotes}
\usepackage{caption}
\usepackage{subcaption}
\usepackage{todonotes}
\usepackage{marginnote}
\let\oldmarginpar\marginpar
\renewcommand\marginpar[1]{\-\oldmarginpar[\raggedleft\footnotesize #1]%
{\raggedright\footnotesize #1}} 
\newtheorem{thm}{Theorem}
\newtheorem{prop}[thm]{Proposition}
\newtheorem{cor}[thm]{Corollary}
\newtheorem{conj}[thm]{Conjecture}
\newtheorem{lemma}[thm]{Lemma}
\newtheorem{claim}[thm]{Claim}
\newtheorem{fact}[thm]{Fact}
\theoremstyle{definition}

\crefname{defi}{Definition}{Definitions}
\theoremstyle{remark}
\newtheorem{rk}{Remark}
\newcounter{cas}

\newtheoremstyle{assert}
  {.5\baselineskip±.2\baselineskip}   
  {.5\baselineskip±.2\baselineskip}   
  {\itshape}  
  {0pt}       
  {\bfseries} 
  {.}         
  {5pt plus 1pt minus 1pt} 
  {(\thmnumber{#2})}          
\theoremstyle{assert}

\newcommand{\restrict}[2]{{#1}_{\left|#2\right.}}

\newcommand{\bsig}{{\bm{\sigma}}}
\newcommand{\bi}{\mathbf{i_0}}
\newcommand{\bA}{\mathbf{A}}
\newcommand{\bC}{\mathbf{C}}
\newcommand{\bS}{\mathbf{S}}

\newcommand{\bX}{\mathbf{X}}
\newcommand{\bY}{\mathbf{Y}}
\newcommand{\bZ}{\mathbf{Z}}
\newcommand{\W}{\mathcal{W}}
\newcommand{\cN}{\mathcal{N}}
\newcommand{\bU}{\mathbf{U}}
\newcommand{\C}{\mathscr{C}}

\newcommand{\pth}[1]{\left(#1 \right)}

\newcommand{\ceil}[1]{\left\lceil #1 \right\rceil}
\newcommand{\pr}[1]{\mathbb{P}\left[ #1 \right]}
\newcommand{\esp}[1]{\mathbb{E}\left[ #1 \right]}
\newcommand{\midbar}{\;\middle|\;}

\newcommand{\sst}[2]{\left\{#1\,:\,#2\right\}}

\newcommand{\threshold}[3]{\langle #1\;?\; #2 : #3 \rangle}

\newcommand{\odd}{\chi_o}
\newcommand{\pcf}{\chi_{\rm pcf}}
\newcommand{\bigO}[1]{\mathcal{O}\pth{#1}}
\newcommand{\bigo}[1]{\mathcal{O}(#1)}

\tikzstyle{vertex} = [draw,fill,shape=circle,node distance=80pt]
\tikzstyle{wertex} = [draw=black,fill=white,shape=circle,node distance=80pt]
\tikzstyle{gertex} = [draw=black,fill=black!25,shape=circle,node distance=80pt]
\tikzstyle{edge} = [fill,opacity=.5,fill opacity=.5,line cap=round, line join=round, line width=50pt]
\pgfdeclarelayer{background}
\pgfsetlayers{background,main}

\makeatletter
\newcommand{\newpar}[1]{%
    \par
    \addvspace{\medskipamount}
    \noindent\textit{#1\@addpunct{.}}\enspace\ignorespaces
    }
\makeatother

\title{New bounds for proper $h$-conflict-free colourings}

\author{Quentin Chuet}
\address{\'Equipe GALaC, LISN (Université Paris-Saclay),
Gif sur Yvette, France.}
\email{quentin.chuet@lisn.fr}

\author{Tianjiao Dai}
\address{School of Mathematics, (East China University of Science and Technology), 
Shanghai, 200237, China. Supported by Natural Science Foundation of Shandong Province (ZR2024QA174).}
\email{tianjiao.dai@sdu.edu.cn}

\author{Qiancheng Ouyang}
\address{Data Science Institute (Shandong University ),
Jinan, China. Supported by Natural Science Foundation of Shandong Province(ZR2024QA107)}
\email{oyqc@sdu.edu.cn}

\author{François Pirot}
\address{\'Equipe GALaC, LISN (Université Paris-Saclay),
Gif sur Yvette, France.}
\email{francois.pirot@lisn.fr}
\begin{document}

\begin{abstract}
A proper $k$-colouring of a graph $G$ is called $h$-conflict-free if every vertex $v$ has at least $\min\, \{h, \deg(v)\}$ colours appearing exactly once in its neighbourhood. Let $\pcf^h(G)$ denote the minimum $k$ such that such a colouring exists. We show that for every fixed $h\ge 1$, every graph $G$ of maximum degree $\Delta$ satisfies $\chi_{\rm{pcf}}^h(G) \le h\Delta + \mathcal{O}(h\log \Delta)$. This expands on the work of Cho \emph{et al.}, and improves a recent result of Liu and Reed in the case $h=1$. We conjecture that for every $h\ge 1$ and every graph $G$ of maximum degree $\Delta$ sufficiently large, the bound $\pcf^h(G) \le h\Delta + 1$ should hold, which would be tight.
When the minimum degree $\delta$ of $G$ is sufficiently large, namely $\delta \ge \max\{100h, 2000\log \Delta\}$, we show that this upper bound can be further reduced to $\chi_{\rm{pcf}}^h(G) \le \Delta + \bigo{\sqrt{h\Delta}}$. This improves a recent bound from Kamyczura and Przyby\l o when $\delta \le \sqrt{h\Delta}$. 
\end{abstract}

\maketitle

\section{Introduction}

For a positive integer $k$, let us denote $[k] \coloneqq \{1,\dots,k\}$ the set of the first $k$ positive integers. A \emph{$k$-colouring} of a graph $G$ is an assignment $\sigma\colon V(G) \rightarrow [k] $. A colouring $\sigma$ of $G$ is \emph{proper} if $\sigma(u) \neq \sigma(v)$ for every edge $uv \in E(G)$. Given a colouring $\sigma$ and a vertex $v \in V(G)$, a \emph{witness} of $v$ is a neighbour $u \in N(v)$ such that the colour $\sigma(u)$ appears exactly once in $N(v)$; we say that $\sigma(u)$ is a \emph{solitary colour} of $v$ in $\sigma$.
The colouring $\sigma$ is \emph{proper conflict-free} (or simply \emph{pcf}) if it is proper and every non-isolated vertex $v$ has a witness. We let $\pcf(G)$ be the smallest integer $k$ such that a pcf $k$-colouring of $G$ exists. 
The notion of pcf colouring was introduced by Fabrici \emph{et al.} \cite{fabrici2023proper}, and was first investigated in the class of graphs of maximum degree $\Delta$ by Caro, Petru\v{s}evski and \v{S}krekovski \cite{caro2023remarks}. Among other results, they obtained the upper bound $\pcf(G) \le 5\Delta(G)/2$, and conjectured that the optimal upper bound should actually be $\Delta(G)+1$.

\begin{conj}[Caro, Petru\v{s}evski, \v{S}krekovski {\cite[Conjecture 6.4]{caro2023remarks}}]\label{conj:pcf}
    If $G$ is a connected graph of maximum degree $\Delta\ge 3$, then $\pcf(G)\le \Delta+1$.
\end{conj}

The condition $\Delta \ge 3$ is needed because the conjectured bound does not hold for cycles. Indeed $\pcf(C_5)=5 = \Delta(C_5)+3$, and for every other cycle $C$ whose length is not a multiple of $3$, $\pcf(C) = 4 = \Delta(C)+2$. The condition that $G$ is connected ensures that there is no isolated vertex.
The conjectured upper bound is best possible. For instance, if $G$ is a graph obtained by subdividing once any subset of edges from the complete graph $K_n$ with $n\ge 4$, then $\pcf(G)=n=\Delta(G)+1$, while $\chi(G)=2$. In particular, when $G$ is obtained by subdividing once every edge of $K_n$, the difference between $\chi(G)$ and $\pcf(G)$ is unbounded.

As a first step towards resolving \cref{conj:pcf}, Cranston and Liu \cite{CrLi24} observed that a simple greedy algorithm provides a pcf $(2\Delta+1)$-colouring of any graph $G$ of maximum degree $\Delta$. They moreoever proved that $\pcf(G) < 1.656\Delta$ when $\Delta$ is large enough. More recently, Liu and Reed \cite{liu2024asymptotically} proved that \cref{conj:pcf} holds asymptotically, by showing $\pcf(G) \le \Delta + \mathcal{O}(\Delta^{2/3}\log \Delta)$ as $\Delta \to \infty$. Our first contribution is to improve this upper bound by eliminating the polynomial factor in the second-order term.

\begin{thm}\label{thm:pcf-log}
    Let $G$ a graph of maximum degree $\Delta$, then $$\pcf(G) \le \Delta + \bigO{\log \Delta}.$$
\end{thm}

Proper conflict-free colourings can be seen as a relaxation of distance-$2$ colourings: the latter require all neighbours of a vertex to have different colours, whereas the former only ask for one neighbour to have a colour distinct from the others. To interpolate between those two notions, one can introduce the notion of \emph{proper $h$-conflict-free $k$-colourings} (or \emph{$h$-pcf $k$-colouring} for short), for some integers $h,k\ge 1$. Those are proper $k$-colourings where every vertex $v$ has at least $\min\, \{h, \deg(v)\}$ solitary colours. In particular, for a graph $G$ of maximum degree $\Delta$, a $1$-pcf $k$-colouring of $G$ is simply a pcf $k$-colouring of $G$, while an $h$-pcf $k$-colouring is a distance-$2$ $k$-colouring of $G$ whenever $h \ge \Delta-1$. 
For every $h\ge 1$, we denote $\pcf^h(G)$ the minimum $k$ such that an $h$-pcf $k$-colouring of $G$ exists. A simple greedy algorithm, inspired by the observation of Cranston and Liu \cite{CrLi24}, yields the following upper bound. We include the proof with a complete description of the greedy algorithm, which will be useful in an alternative form in the rest of the paper.

\begin{prop}\label{prop:h-pcf-greedy}
    Let $G$ be a graph of maximum degree $\Delta$, and $h \ge 1$ an integer. Then $$\pcf^h(G) \le (h+1)\Delta + 1.$$
\end{prop}
\begin{proof}
    Fix any linear ordering $(v_1, \dots, v_n)$ of $V(G)$. For a vertex $v \in V(G)$, let $\W(v)$ be the first $\min\,\{h, \deg(v)\}$ neighbours of $v$ in the ordering. For $i$ from $1$ to $n$, colour $v_i$ with the first colour which does not appear in $N(v_i)$, nor in $\W(u)$ for $u \in N(v_i)$. Let $\sigma$ be the colouring obtained at the end of this procedure. One easily checks that $\sigma$ is proper. Furthermore, for all $v \in V$, the vertices of $\W(v)$ are witnesses of $v$ (the fact that $\W(v)$ is coloured before $N(v)\setminus \W(v)$ is important), therefore $\sigma$ is $h$-pcf. Observe that the number of colour constraints when colouring $v_i$ is at most $(h+1)\deg(v_i)$, therefore $\sigma$ uses at most $(h+1)\Delta + 1$ colours.
\end{proof}

\begin{rk}
    With the additional assumption that $G$ is $d$-degenerate, we obtain $\pcf^h(G) \le h\Delta + d + 1$.
\end{rk}

Cho \emph{et al.} \cite{cho2025brooks} improved this upper bound by reducing the number of colours needed by $2$.

\begin{thm}[Cho, Choi, Kwon, Park, 2025]\label{thm:pcf-cho}
    Let $G$ be a graph of maximum degree $\Delta$, and let $1 \le h \le \Delta - 2$ be an integer. Then $\pcf^h(G) \le (h+1)\Delta - 1$.
\end{thm}

For the specific case $h = \Delta - 2$, they exhibit a graph $G$ of maximum degree $\Delta$ arbitrarily large such that $\pcf^{h}(G) = (h+1)(\Delta - 1)$. Note that this last term is also equal to $h\Delta + 1$. We consider the opposite end of the spectrum by considering $h$ a fixed constant, and make an asymptotic improvement on \cref{thm:pcf-cho} in that case.

\begin{thm}\label{thm:pcf-upper}
    Let $h\ge 1$ be a fixed integer. If $G$ is a connected graph of maximum degree $\Delta$, then $$\pcf^h(G)\le h\Delta + \bigO{h \log\Delta}.$$
\end{thm}

\cref{thm:pcf-log} is a direct corollary of \cref{thm:pcf-upper} with $h=1$. The precise dependency hidden behind the term $\bigo{h\log\Delta}$ is explicitly stated in \cref{cor:h-pcf}.\\

The notion of pcf colourings can also be seen as a strengthening of \emph{odd colourings}, cf. \cite{petruvsevski2022colorings,CCKP22+,dai2023new}.
For a graph $G$ and an integer $h \ge 1$, an \emph{$h$-odd colouring} of $G$ is a proper colouring such that every vertex $v$ observes at least $\min\, \{h, \deg(v)\}$ colours in its neighbourhood that appear an odd number of times. Let $\odd^h(G)$ be the minimum number of colours required for an $h$-odd colouring of $G$, and $\odd(G) \coloneqq \odd^1(G)$. Clearly, $\odd^h(G) \le \pcf^h(G)$ for every $h\ge 1$. A weaker version of \cref{conj:pcf} was first formulated in \cite{CPS22+} for the odd chromatic number of graphs of maximum degree $\Delta$, and it was shown in \cite[Theorem 3]{dai2023new} that $\odd(G) \le \Delta + \bigO{\log\Delta}$. In that sense, \cref{thm:pcf-log} is also a strengthening of this result.

If $h$ is a fixed integer and $G$ is a graph of maximum degree $\Delta$, then we have $\odd^h(G) \le h\Delta + \bigO{h\log\Delta}$ as a direct consequence of \cref{thm:pcf-upper}. The first-order term of this upper bound was shown to be tight in \cite[Proposition 28]{dai2023new}.

\begin{prop}[Dai, Ouyang, Pirot, 2024]\label{prop:lower-bound}
    Let $h \ge 1$ be a fixed integer. For infinitely many values of $\Delta$, there exists a graph $G$ of maximum degree $\Delta$ and minimum degree $h+1$ such that 
    \[\odd^h(G) \ge h\Delta + 1.\]
\end{prop}

In turn, \cref{prop:lower-bound} demonstrates that \cref{thm:pcf-upper} is tight up to the second-order term, i.e. for every fixed integer $h \ge 1$, there exists a graph of maximum degree $\Delta$ arbitrarily large such that $\pcf^h(G) \ge h\Delta + 1$. In light of these observations, we propose an extended version of \cref{conj:pcf}.
    
\begin{conj}
    Let $h \ge 1$ be an integer, and $G$ be a graph of maximum degree $\Delta$ sufficiently large. Then $$\pcf^h(G) \le h\Delta + 1.$$
\end{conj}

The lower bound provided by \cref{prop:lower-bound} relies on a graph construction with many vertices of small degree (namely, of degree $h+1$). When there is an additional restriction on the minimum degree of $G$, it is possible to derive smaller upper bounds on $\pcf^h(G)$. 
In particular, Liu and Reed \cite{LiRe24+} have recently proved that \cref{conj:pcf} holds for regular graphs in a strong sense. Their result can be stated as follows.

\begin{thm}[Liu, Reed, 2024+]
\label{thm:peaceful}
    For every graph $G$ of maximum degree $\Delta$ and minimum degree $\delta > \frac{8000}{8001}\Delta$, letting $h\coloneqq \delta - \frac{8000}{8001}\Delta$, one has
    \[ \pcf^h(G) \le \Delta+1. \]
\end{thm}

The condition on the minimum degree in \cref{thm:peaceful} is of course very restrictive. One could wonder what happens when this condition is relaxed.
This has been considered by Kamyczura and Przyby\l o \cite{kamyczura2024conflict}, who showed the following.

\begin{thm}[Kamyczura, Przyby\l o, 2024]\label{thm:kamyczura}
    Let $G$ be a graph of maximum degree $\Delta$ and minimum degree $\delta$. For every integer $h$ satisfying $20 \log \Delta \le h \le \delta/75$, one has
    \[\pcf^h(G) < \Delta + \frac{30 h \Delta}{\delta}.\]
\end{thm}

The condition on $h$ implies that $\delta > 1500\log\Delta$.
When $h< 20\log \Delta$, an upper bound on $\pcf^h(G)$ can be obtained by relying on the monotonicity of $\pcf^h(G)$ in $h$. 
One has $\pcf^h(G) \le \pcf^{\ceil{20 \log \Delta}}(G)$, and can obtain an upper bound on $\pcf^h(G)$ with an application of \Cref{thm:kamyczura}.
A direct consequence of \Cref{thm:kamyczura} is that, when $h$ is a fixed integer and $\delta \gg \log\Delta$, one has $\pcf^h(G) \le \Delta + \bigo{\Delta \log\Delta/\delta}$. We improve this bound when the minimum degree is not too large.

\begin{thm}\label{thm:pcf-large-mindeg}
    Let $G$ be a graph of maximum degree $\Delta$ and minimum degree $\delta$. For every integer $h$ satisfying $20 \log \Delta \le h \le \delta/100$, one has
    \[\pcf^h(G) \le \Delta + \bigO{\sqrt{h\Delta}}.\]
\end{thm}

The upper bound provided by \cref{thm:pcf-large-mindeg} is smaller than that of \cref{thm:kamyczura} in the regime $\delta = \bigo{\sqrt{h\Delta}}$, and performs best at $\delta = 100h$. 
However, we suspect that there is still room for improvement in that regime. For matters of comparison, it was proved in \cite[Corollary 6.3]{dai2023new} that if $3\log\Delta \le h \le \delta/2$, then $\odd^h(G) \le \Delta + \bigo{h}$, and we conjecture that a similar bound should hold for $h$-pcf colourings as well.

\begin{conj}
    \label{conj:h-pcf}
    There exist absolute constants $c_1,c_2 > 0$ such that the following holds. Let $G$ be a graph of maximum degree $\Delta$ and minimum degree $\delta$. For every integer $h$ satisfying $c_1 \log \Delta \le h \le c_2 \,\delta$, one has
    \[\pcf^h(G) \le \Delta + \bigO{h}.\]
\end{conj}

\subsection{Notations and outline}
Given a graph $G$ clear from the context and an integer $d$, we denote $V_{\le d} \subseteq V(G)$ the set of vertices of degree at most $d$ in $G$, and $V_{> d}\coloneqq V(G) \setminus V_{\le d}$. Given a subset of vertices $U \subseteq V(G)$, and two real values $a,b$, we denote 

\[ 
\threshold{U}{a}{b}\colon \left\{
\begin{array}{r l}
    V(G) & \to \{a, b\} \\
    v &\mapsto \begin{cases}
        a &  \mbox{if $v \in U$;}\\
        b & \mbox{otherwise}
    \end{cases}         
\end{array}\right.\]
the threshold function equal to $b + (a-b)\mathds{1}_U$.

\medskip 

The paper is organised as follows. In \cref{sec:proba}, we provide probabilistic tools that we will rely on in our proofs. We consider general graphs in \cref{sec:general} and prove \cref{thm:pcf-upper}. In \cref{sec:min-degree-condition} we consider graphs of constrained minimum degree and prove \cref{thm:pcf-large-mindeg}.

\section{Probabilistic tools}\label{sec:proba}

Let $B(n,p)$ denote a random variable that follows the binomial distribution of parameters $n$ and $p$. We will rely on the following form of Chernoff's bounds that in particular provides large concentration inequalities for binomial random variables.

\begin{lemma}
    [Chernoff's bounds]\label{lem:chernoff}
    Let $\bX_1, \ldots, \bX_n$ be i.i.d. $(0,1)$-valued random variables, and let $\bS_n \coloneqq \sum_{i=1}^n \bX_i$. Let us write $\mu\coloneqq \esp{\bS_n}$. Then
    \begin{enumerate}[label={\rm(\roman*)}]
        \item for every $0 <\delta < 1$,
         \[ \pr{\bS_n \le (1-\delta)\mu } \le e^{-\delta^2\mu/2};\]
         \label{it:upper}
        \item for every $\delta > 0$, 
        \[ \pr{\bS_n \ge (1+\delta)\mu } \le e^{-\delta^2\mu/(2+\delta)}.\]
        \label{it:lower}
    \end{enumerate}
\end{lemma}

In our proofs, we shall mostly use Chernoff's bounds in the following setting. Those respectively follow from \Cref{lem:chernoff}\ref{it:upper} and \ref{it:lower} with $\delta \coloneqq 1/2$ and $\delta \coloneqq 1/4$.

\begin{lemma}\label{lem:chernoff2}
    Let $n \ge 1$ and $p \in (0,1)$. Then \begin{enumerate}[label={\rm (\roman*)}]
        \item $\pr{B(n,p) \le \frac{1}{2} np} \le e^{-np/8}$,\label{it:left} 
        \item $\pr{B(n,p) \ge \frac{5}{4}np} \le e^{-np/36}$.\label{it:right} 
    \end{enumerate}
\end{lemma}

If $p = \bigO{1/n}$, Chernoff's bounds are not precise enough to obtain tight tail bounds on $B(n,p)$. Instead, we shall rely on the following well-known estimate that follows from a simple union bound (we include the proof for completeness). 

\begin{lemma}\label{lem:binomial-tail-simple}
    Let $0 \le  t \le n$ be an integer. Then $\pr{B(n,p) \ge t} \le \pth{enp/t}^{t}$.
\end{lemma}

\begin{proof}
    Let $\bY \coloneqq \sum_{i=1}^n \bX_i$, where the variables $\{\bX_i\}$ are independent and follow a Bernoulli distribution of parameter $p$. Then $\bY \sim B(n,p)$, and $\bY \ge t$ iff there exists a $t$-set $I\in \binom{[n]}{t}$ such that $\bX_i=1$ for every $i\in I$. For a fixed $I$, this happens with probability $p^t$, so using a union bound we obtain
    \[ \pr{B(n,p)\ge t} = \pr{\bY \ge t} \le \binom{n}{t} p^t \le \pth{\frac{enp}{t}}^t.\qedhere\]
\end{proof}

Suppose we have binary random variables $\bX_1,\dots,\bX_n$ such that $\pr{\bX_i = 1} \ge p$ regardless of the realisation of $\bX_1,\dots,\bX_{i-1}$, for all $i \in [n]$. These variables are not necessarily independent, but we expect $\sum \bX_i$ to be at least as large as a binomial variable of parameters $n,p$. We formalise this idea in \cref{lem:cond-binom}.

\begin{lemma}\label{lem:cond-binom}
    Let $\bX = (\bX_1,\dots,\bX_n)$ be a sequence of binary random variables such that 
    \begin{equation}
        \tag{$\star$}
        \label{eq:conditional-proba}
        \pr{\bX_i = 1 \midbar \bX_1,\dots, \bX_{i-1}} \ge p
    \end{equation}
    for all $i\in [n]$. Then $\pr{\sum \bX_i \le t} \le \pr{B(n,p) \le t}$.
\end{lemma}

\begin{proof}
    We shall use a classical coupling argument. Given a sequence $A = (A_1, \dots, A_n)$ and an integer $i \in [n]$, we define $A_{< i} \coloneqq (A_1,\dots,A_{i-1})$. Let $\bU = (\bU_1,\dots,\bU_n)$ be a sequence of independent continuous uniform variables in the interval $[0,1]$. We define a sequence of binary random variables $\bX' = (\bX_1',\dots,\bX_n')$ on the probability space of $\bU$. For $x \in \{0,1\}^{n}$ and $i \in [n]$, if $\bX'_{<i} = x_{<i}$, let 
    \[
        \bX_i' \coloneqq \begin{cases}
        1 & \mbox{if $\bU_i \le \pr{\bX_i = 1 \midbar \bX_{<i} = x_{<i}}$;} \\
        0 & \mbox{otherwise.}
        \end{cases}\]

    We first show that $\bX \sim \bX'$. 
    Let $x = (x_1, \dots, x_n) \in \{0,1\}^n$. Observe that, by construction, $\pr{\bX'_i = x_i \midbar \bX'_{<i} = x_{<i}} = \pr{\bX_i = x_i \midbar \bX_{<i} = x_{<i}}$ for all $i \in [n]$. Using this observation and the chain rule, we deduce that 
    \[\pr{\bX = x} = \prod_{i = 1}^n \pr{\bX_i = x_i \midbar \bX_{< i} = x_{<i}} = \prod_{i = 1}^n \pr{\bX'_i = x_i \midbar \bX'_{< i} = x_{<i}} = \pr{\bX' = x}.\] 
 
    Therefore, $\pr{\sum \bX_i \le t} = \pr{\sum \bX'_i \le t}$. We now define a second sequence $\bY = (\bY_1,\dots,\bY_n)$ of binary random variables on the same probability space. For $i \in [n]$, let
    \[\bY_i\,\, \coloneqq \begin{cases}
            1 &\mbox{if $\bU_i \le p$;}\\
            0 & \mbox{otherwise.}
           \end{cases}
    \]
    Clearly, $\bY$ is a sequence of independent Bernoulli variables of parameter $p$, hence $\sum \bY_i \sim B(n,p)$. We easily check that $\bY_i = 1$ implies $\bX_i' = 1$ for all $i\in[n]$, thus, $\sum \bX_i' \ge \sum \bY_i$. It follows that \[\pr{\sum \bX_i \le t} = \pr{\sum \bX_i' \le t} \le \pr{\sum \bY_i \le t}.\qedhere\]
\end{proof}

Our proofs rely on the Symmetric Lov\'{a}sz Local Lemma; we present the refined version from Shearer \cite{shearer1985problem} for convenience, although the original version due to Lov\'{a}sz would be sufficient for our needs.

\begin{lemma}[Lov\'{a}sz Local Lemma]\label{lem:lll}
Let $\mathcal{B} = \{B_1, \ldots, B_n\}$ be a finite set of random (bad) events, each of which happens with probability at most $p$. Suppose that, for every $i\in [n]$, $B_i$ is independent from all but at most $d$ other events.

If $epd\le 1$, then 
\(\pr{\bigcap_{i\in [n]} \limits \overline{B_i}} > 0.\)
\end{lemma}

In some applications, \cref{lem:lll} is not sufficient: for these, we shall use the following lopsided version of the Symmetric Lov\'{a}sz Local Lemma (see e.g. \cite{Ber19}).

\begin{lemma}[Lopsided Lov\'{a}sz Local Lemma]\label{lem:llll}
Let $\mathcal{B} = \{B_1, \ldots, B_n\}$ be a finite set of random (bad) events, and let $d$ be a fixed integer. Suppose that, for every $i\in [n]$, there is a set $\Gamma(i)\subseteq [n]$ of size at most $d$ such that, for every $Z\subseteq [n]\setminus\Gamma(i)$,
\[\pr{B_i \; \middle| \; \bigcap_{j\in Z}\overline{B_j}}\le p.\]

If $4pd\le 1$, then $\pr{\bigcap_{i\in [n]} \limits \overline{B_i}}>0$.
\end{lemma}

\section{A general upper bound for  \texorpdfstring{$\pcf^h$}{the h-conflict-free chromatic number}}\label{sec:general}

We recall that, in a proper colouring $\sigma$ of a graph $G$, a solitary colour of a vertex $v\in V(G)$ is a colour of $\sigma$ having a unique occurrence in $N(v)$, and each $u\in N(v)$ coloured with a solitary colour of $v$ is called a witness of $v$.

\subsection{Overview of the proof.}

We mix two techniques in order to ensure that the conflict-free constraints are all satisfied. Those that arise from small neighbourhoods can be dealt with greedily, while those that arise from large neighbourhoods can be dealt with randomly.

Therefore, we perform a two-step colouring of the graph $G$. First, we shall use a modified version of the greedy algorithm presented in the proof of \cref{prop:h-pcf-greedy} to construct a proper precolouring $\sigma_0$ of $G$ such that all vertices of ``small degree'' have the required number of witnesses, while those of ``large degree'' may still miss one witness each. During the second step, we shall randomly recolour a fraction of the vertices with new colours in order to create one extra witness for each vertex of large degree. 
For that, we use a Rödl Nibble approach that takes inspiration from the work of Kamyczura and Przyby\l o \cite{kamyczura2024conflict}, but incorporates new ideas to ensure that witnesses from the precolouring $\sigma_0$ are preserved, and also relies on a different procedure which we believe is conceptually simpler to understand and analyse.

\subsection{The proof}

Given a graph $G$ and a real number $d \ge 0$, we recall that $V_{\le d} \subseteq V(G)$ denotes the set of vertices with degree at most $d$, and $V_{> d} \coloneqq V(G) \setminus V_{\le d}$ the set of vertices with degree greater than $d$. 
Given a function $f\colon V(G)\to \mathbb{Z}_{\ge 0}$, a proper $k$-colouring $\sigma \colon V(G) \to [k]$ is said to be \emph{$f$-conflict-free} if every vertex $v \in V(G)$ has at least $\min\, \{f(v),\deg(v)\}$ witnesses in $\sigma$.
In our work, we will exclusively rely on a specific kind of threshold function for $f$. Let $h\ge h_0 \ge 0$ be integers, and $U\subseteq V(G)$; we recall that
\[ 
\threshold{U}{h}{h_0}\colon \left\{\begin{array}{r l}
    V(G) & \to \{h, h_0\} \\
    v &\mapsto \begin{cases}
        h &  \mbox{if $v \in U$;}\\
        h_0 & \mbox{otherwise}
    \end{cases}         
\end{array}\right.\]
denotes the threshold function $h_0 + (h-h_0) \mathds{1}_U$.

By slightly modifying the greedy algorithm presented in the proof of \cref{prop:h-pcf-greedy}, we obtain a colouring which uses $\Delta$ fewer colours, at the cost of granting one fewer witness for vertices of large degree.

\begin{lemma}\label{lem:partial-pcf}
    Let $h \ge 1$ be an integer, and $G$ a graph of maximum degree $\Delta$. Then there exists a proper $(h\Delta+1)$-colouring $\sigma_0$ of $G$ such that $\sigma_0$ is $\threshold{V_{\le \frac{h}{h+1}\Delta}}{h}{ h-1}$-conflict-free.
\end{lemma}
\begin{proof}
    Let $V^+ \coloneqq V_{> \frac{h}{h+1}\Delta}$ and $V^- \coloneqq V_{\le \frac{h}{h+1}\Delta}$. 
    
    Fix a linear ordering $(v_1, \dots, v_n)$ of $V$ such that all vertices of $V^-$ appear after those of $V^+$ (e.g. order the vertices by decreasing degree in $G$). For a vertex $v \in V^-$, let $\W(v)$ be the first $\min\, \{h, \deg(v)\}$ neighbours of $v$ in the ordering, and for $v \in V^+$, let $\W(v)$ be the first $\min\, \{h-1,\deg(v)\}$ neighbours of $v$ in the ordering. Colour sequentially each vertex $v_i$, for $i$ from $1$ to $n$, with the first colour that does not appear in $N(v_i)$, nor in $\W(u)$ for all $u \in N(v_i)$. Let $\sigma_0$ be the colouring obtained at the end of this procedure. One easily checks that $\sigma_0$ is proper. Furthermore, for all $v \in V$, the vertices of $\W(v)$ are witnesses of $v$ (the fact that $\W(v)$ is coloured before $N(v)\setminus \W(v)$ is important), thus $\sigma_0$ is $\threshold{V^-}{h}{h-1}$-conflict-free.
    
    We now bound the number of colours used in $\sigma_0$. Consider the number of colour constraints when colouring $v_i$, for $i\in [n]$.
    \begin{itemize}
        \item If $v_i \in V^+$, then each neighbour $u \in N(v_i)$ forbids at most $h$ colours for $v_i$: indeed, either $u \in V^+$ and the colour constraints come from $u$ and $\W(u)$, or $u \in V^-$ in which case $u$ has not been coloured yet (by definition of our ordering) and only the colours of $\W(u)$ must be avoided. In total, at most $h\Delta$ colours are forbidden.
        \item If $v_i \in V^-$, then $\deg(v_i) \le \frac{h}{h+1}\Delta$, and each neighbour $u \in N(v_i)$ forbids at most $h+1$ colours for $v_i$. In total, at most $h\Delta$ colours are forbidden.
    \end{itemize}
    Therefore, $\sigma_0$ uses at most $h\Delta + 1$ colours.
\end{proof}

\newcommand{\add}{\eta}

Given a graph $G$ of maximum degree $\Delta$, an integer $h\ge 1$, and a proper $k$-colouring $\sigma_0$ of $G$ which is $\threshold{V_{\le \frac{h}{h+1}\Delta}}{h}{h-1}$-conflict-free, as provided by \cref{lem:partial-pcf}, we want to make small alterations to $\sigma_0$ in such a way that
\begin{enumerate}[label=(\roman*)]
    \item vertices in $V_{> \frac{h}{h+1}\Delta}$ gain one additional witness;
    \item all vertices keep their witnesses from the colouring $\sigma_0$.
\end{enumerate}
We do this by introducing $m$ additional colours, and thus obtain an $h$-pcf $(k+m)$-colouring. This is possible for some $m = \bigO{h\log\Delta}$; this is proved in \cref{lem:pcf-precoloured} using a Rödl Nibble approach.

\begin{lemma}\label{lem:pcf-precoloured}
    Let $G$ be a graph of maximum degree $\Delta \ge 13000$. Let $h$, $h_0$, $d$, $\add$ be non-negative integers such that $h_0 \le \frac{d}{4}$ and $25 \log\Delta \le \eta \le \frac{d}{20e}$. Suppose there exists a proper $k$-colouring $\sigma_0$ of $G$ that is $\threshold{V_{\le d}}{h}{h_0}$-conflict-free. Let $m \coloneqq \ceil{4e(h+2)\Delta \add / d}$. 
    Then there exists a proper $(k+m)$-colouring $\sigma$ of $G$ that is $\threshold{V_{\le d}}{h}{h_0 + \add}$-conflict-free.
\end{lemma}

\newcommand{\witnessneighborhood}{\cN}

\begin{proof}
For $v \in V_{\le d}$, let $\W(v)$ be an arbitrary subset of $\min\, \{h,\deg(v)\}$ witnesses of $v$ in $\sigma_0$. For $v \in V_{> d}$, let $\W(v)$ be an arbitrary subset of $h_0$ witnesses of $v$ in $\sigma_0$ (here we use the assumption that $h_0 \le d$). For every subset $S\subseteq V(G)$ we denote $\W(S) \coloneqq \bigcup_{u \in S}\W(u)$, and for every $u \in V(G)$ we denote $\witnessneighborhood(u) \coloneqq N(u) \cup \W(N(u)) \setminus \{u\}$.

We define $p \coloneqq \frac{1}{(h + 2)\Delta}$ and execute the following procedure, which first initialises subsets $A_1\dots,A_m$ by randomly sampling $V(G)$ with independent probability $p$, and then proceeds in $m$ rounds to construct subsets $C_1,\dots,C_m$ which will serve as new colour classes in the colouring $\sigma_0$. We note that the randomness of the procedure is entirely contained in the \textbf{Initialisation} bloc, which simply constructs $m$ independent binomial random subsets of $V(G)$. The \textbf{fail-safe} bloc could be replaced with an interruption of the procedure, but its presence makes the analysis easier by preserving some crucial invariants needed in the rest of the proof.

\setlength{\interspacetitleruled}{0pt}%
\setlength{\algotitleheightrule}{0pt}%
\begin{algorithm}[H]
\textbf{Initialisation: }\For{$i \in [m]$}{
    $A_i \gets \emptyset$\;
    \For{$u \in V(G)$}{
        Add $u$ to $A_i$ independently with probability $p$.\;
    }
}
$i_0\gets 1$\;
\For{$i \in [m]$}{
    $S_i \gets \{u \in V(G) : \witnessneighborhood(u) \cap A_i \ne \emptyset\}\ \cup\ \bigcup_{j = i_0}^{i-1} C_j$\;
    $C_i \gets A_i \setminus S_i$.\;
    
    \textbf{fail-safe:} \If{there exists $v \in V_{>d}$ such that $|N(v) \cap \bigcup_{j = i_0}^{i} C_j| > \frac{\deg(v)}{4}$}{
        $i_0 \gets i+1$.\;
    }
}
\end{algorithm}\SetNlSty{texttt}{(}{)}

For $i \in [m]$, we respectively let $\bA_i$, $\bS_i$, $\bC_i$, and $\bi$ be the random values of the sets $A_i$, $S_i$, $C_i$ and the integer $i_0$ at the end of the procedure. Let $\bsig$ be the random colouring obtained from $\sigma_0$ after recolouring $\bC_i$ with the colour $k + i$ for each $i \in \{\bi, \dots, m\}$. Note that there is no ambiguity in the definition of $\bsig$ since the sets $\bC_{\bi}, \dots, \bC_m$ are disjoint by construction.

The \textbf{fail-safe} serves no practical use: we want to keep the variable $i_0$ equal to $1$ throughout the procedure, but in order to simplify the analysis, we deterministically ensure that at the start of each round $i \in [m]$, vertices in $V_{>d}$ have sufficiently few neighbours in $\bigcup_{j = i_0}^{i-1} C_j$: if the \textbf{fail-safe} is activated, we essentially forget all the work done up to that point and continue the procedure. This technicality was made evident in the work of Kamyczura and Przyby\l o \cite{kamyczura2024conflict}, who proposed a different approach by introducing extra vertices at the start of each round, and then argued that these extra vertices are in fact unnecessary. For $v \in V_{>d}$, let $B_v$ be the (bad) event that more than $\deg(v)/2$ neighbours of $v$ belong in $\bA_1 \cup \dots \cup \bA_m$. Observe that, if $B_v$ does not occurs, then in particular $|N(v) \cap (\bC_1 \cup \dots \cup \bC_m)| \le \frac{1}{2}\deg(v)$. Therefore, if none of the events $\{B_v\}_{v \in V_{>d}}$ occur, then the \textbf{fail-safe} is never activated, and thus $\bi = 1$.

\begin{claim}\label{claim:reset}
    For $v \in V_{>d}$, $\pr{B_v} \le \Delta^{-7}$
\end{claim}
\begin{proof}
    Let $\bZ_v \coloneqq \sum_{i=1}^{m}|\bA_i \cap N(v)| \sim B(m \deg(v), p)$. We have $\frac{5}{4}m \deg(v) p < \frac{\deg(v)}{4}$ by the assumption $\eta \le \frac{d}{20e}$, therefore,
    
    \[\pr{B_v} \le \pr{\bZ_v \ge \frac{\deg(v)}{4}} \le \pr{B\pth{m \deg(v), p} \ge \frac{5}{4}m \deg(v)p} \;\le\; e^{-m \deg(v)p/36} < \Delta^{-7},\]
    where we have used \cref{lem:chernoff2}\ref{it:right} together with the fact that $m\deg(v)p/36 \ge mdp/36 \ge e\eta/9 > 7\log\Delta$.
\end{proof}

For $i\in [m]$, the procedure ensures that two adjacent vertices cannot both be present in $\bC_i$, so $\bC_i$ is an independent set of $G$. Furthermore, for $v \in V$, if $u \in \W(v)$ and $u \in \bC_i$, then the procedure ensures that $N(v)\cap \bC_i = \{u\}$, hence $u$ remains a witness of $v$ in $\bsig$.

\begin{fact}\label{fact:guaranteed-witnesses}
    $\bsig$ is a proper colouring, and for all $v \in V$, $\W(v)$ is a set of witnesses of $v$ in $\bsig$.
\end{fact}

For $v \in V_{>d}$, let $\bX_v^i$ be the indicator variable for the event ``$N(v) \cap \bC_i = \{u\}$ for some $u \notin \W(v)$''. The following fact is easily deduced from the disjointness of the sets $\bC_{\bi}, \dots, \bC_m$ and the previous fact.

\begin{fact}\label{fact:keep-plus}
    Let $v \in V_{>d}$. At the end of the procedure, $v$ has at least $h_0 + \displaystyle \sum_{i=\bi}^m \bX_v^i$ witnesses in $\bsig$.
\end{fact}

For every $v\in V_{>d}$, let $\bX_v \coloneqq \sum_{i=1}^{m} \bX_v^i$. If $\bi = 1$ and $\bX_v \ge \add$, then $v$ has at least $h_0 + \add$ witnesses in $\bsig$, as desired. Let $E_v$ be the (bad) event that $\bX_v < \add$.

\begin{claim}\label{claim:witnesses}
    Let $v \in V_{>d}$. Then $\pr{E_v} \le \Delta^{-6.25}$.
\end{claim}

\begin{proof}
    Consider $i \in [m]$. Let us fix the realisation of $\bA_1,\dots,\bA_{i-1}$, and the resulting values of $\bC_1,\dots,\bC_{i-1}$. Let $C'$ be the resulting value of the term ``$\bigcup_{j = i_0}^{i-1} C_j$'' at the start of round $i$. The distribution of $\bA_i$ is unaffected by this conditioning, i.e. every vertex  $u\in V(G)$ is still present in $\bA_i$ independently with probability $p$. Let $U \coloneqq N(v) \setminus (\W(v)\cup C')$ and let $t \coloneqq |U|$. Because of the \textbf{fail-safe}, $C'$ contains at most $\deg(v)/4$ neighbours of $v$, and $|\W(v)| = h_0 \le \frac{d}{4}$, therefore $t \ge \deg(v) - h_0 -\frac{\deg(v)}{4} \ge \frac{d}{2}$.
    
    The probability that $N(v) \cap \bA_i = \{u\}$ with $u \in U$ is equal to $t p (1-p)^{\deg(v)-1} \ge\frac{dp}{2}(1-p)^{\Delta-1}$. Suppose such a vertex $u \in U$ exists, i.e. we are conditioning on the event ``$N(v) \cap \bA_i = \{u\}$ for some $u \in U$''. By independence of the random choices, the distribution of $\bA_i$ outside of $N(v)$ remains unaffected, i.e. every vertex $x \in V \setminus N(v)$ is present in $\bA_i$ independently with probability $p$. We now determine a lower bound on the probability that $u \in \bC_i$, which is equivalent to $u \notin \bS_i$ since $u \in \bA_i$. By definition of $U$, $u \notin C'$, therefore $u \notin \bS_i$ if and only if $\witnessneighborhood(u)\cap \bA_i = \emptyset$. We already know that $N(v)\cap \bA_i = \{u\}$, so this happens with probability $(1-p)^{|\witnessneighborhood(u)\setminus N(v)|} \ge (1-p)^{(h + 1)\Delta}$. Therefore,

    \begin{equation}\label{eq:round-i}
        \pr{\bX_v^i = 1} \ge \frac{dp}{2}(1-p)^{(h + 2)\Delta-1} > \frac{dp}{2e},
    \end{equation}
    where we have used the fact that $(1-\frac{1}{x})^{x-1} > \frac{1}{e}$ for all $x > 1$. Let $q \coloneqq \frac{dp}{2e}$. We have shown that $\pr{\bX_v^i = 1} > q$ regardless of the realisation of $\bA_1,\dots,\bA_{i-1}$, and so in particular, 

    \[\pr{\bX_v^i = 1 \midbar \bX_v^1,\dots,\bX_v^{i-1}} > q.\]
    
    By \cref{lem:cond-binom}, we obtain $\pr{\sum \bX_v^i < \add} \le \pr{B(m,q) < \add}$. We have defined $m$ such that $mq \ge 2\eta$, and since $\eta \ge 25\log\Delta$ we have $\frac{mq}{8} \ge 6.25\log\Delta$; by \cref{lem:chernoff2}\ref{it:left}, we conclude that $\pr{E_v} \le \pr{B(m,q) < \add} \le \Delta^{-6.25}$.
\end{proof}

For $v\in V_{>d}$, the bad events $B_v$ and $E_v$ are determined by the realisation of $(\bA_1,\ldots,\bA_m)$ within the distance-$3$ neighbourhood of $v$, therefore these bad events are independent from all but at most $2\Delta^6$ other such bad events. By \cref{lem:lll}, with nonzero probability, none of the events $B_v$ or $E_v$ occur. In that case, $\bi = 1$ and $\bX_v \ge \add$ for every $v \in V_{> d}$, so by \cref{fact:guaranteed-witnesses} and \cref{fact:keep-plus}, we conclude that the colouring $\bsig$ is proper and $\threshold{V_{\le d}}{h}{h_0 + \add}$-conflict-free.
\end{proof}

\cref{thm:pcf-upper} is a direct consequence of the following corollary, which provides explicit constants.

\begin{cor}\label{cor:h-pcf}
    Let $G = (V,E)$ be a graph of maximum degree $\Delta \ge 28000$ and $1\le h \le \Delta$ an integer. Then \[\pcf^h(G) \le h\Delta + 100e(h+5)\log\Delta\] 
    
\end{cor}
\begin{proof}
    We may assume that $h \le \frac{\Delta}{8}$, since otherwise $100e(h+5)\log\Delta > \Delta$ and we already know that $\pcf^h(G) \le (h+1)\Delta + 1$ by \cref{prop:h-pcf-greedy}.
    
    By \cref{lem:partial-pcf}, there exists a proper $\threshold{V_{\le \frac{h}{h+1}\Delta}}{h}{h-1}$-conflict-free $(h\Delta+1)$-colouring $\sigma_0$ of $G$. We apply \cref{lem:pcf-precoloured} with $k\coloneqq h\Delta + 1$, $d \coloneqq \frac{h}{h+1}\Delta \ge \frac{\Delta}{2}$, $h_0 \coloneqq h-1$, and $\add \coloneqq \ceil{25\log\Delta}$. We have $h_0 \le d/4$ by the assumption $h \le \frac{\Delta}{8}$, and we have $\eta \le \frac{\Delta}{40e} \le \frac{d}{20e}$ by the assumption $\Delta \ge 28000$. Then, we obtain a proper $(k + m)$-colouring $\sigma$ that is $\threshold{V_{\le d}}{h}{h-1+\eta}$-conflict-free, where $m = \ceil{4e(h+2)\Delta \eta / d}$. A quick computation shows that $k+m \le h\Delta + 100e(h+5)\log\Delta$.
\end{proof}

\section{Bounding \texorpdfstring{$\pcf^h$}{the h-conflict-free chromatic number} for graphs of constrained minimum degree}
\label{sec:min-degree-condition}

The first-order term $h\Delta$ of the upper bound in \cref{thm:pcf-upper} is tight, by \cref{prop:lower-bound}: the lower bound $h\Delta+1$ is certified by a graph construction of minimum degree $h+1$. In fact, such graphs are highly constrained due to the requirement that neighbourhoods of size at most $h+1$ must be rainbow in an $h$-pcf colouring. If the minimum degree $\delta$ is sufficiently large with respect to $h$ and $\Delta$, then one would expect more flexibility when dealing with small neighbourhoods.
Indeed the upper bound on $\pcf^h(G)$ from \cref{thm:kamyczura} established by Kamyczura and Przyby\l o in \cite{kamyczura2024conflict} shows that only a little more than $\Delta$ colours are required for an $h$-pcf colouring in that case, with the extra assumption that $\delta$ is sufficiently large with respect to $\Delta$. 

We wish to drop the latter extra assumption, and consider a minimum degree constraint as weak as possible that still ensures that $\pcf^h$ is $(1+o(1))\Delta$. 
To do so, we again rely on a two-step colouring process, that takes advantage of having large minimum degree. 
In the previous section, the first step described by \cref{lem:partial-pcf} provides exactly $h$ witnesses for vertices of small degree; as a consequence, during the second step described in \cref{lem:pcf-precoloured}, we need to ensure that no witness created during the first step is destroyed, which is very costly.
In contrast, the extra minimum degree condition lets us replace the deterministic greedy first step with a random one that creates more witnesses for vertices of small degree, namely $2h$ instead of $h$, so that during the second step we may destroy a fraction of those witnesses.
For that, we require that $\delta \ge 16h$, and in \cref{partial-pcf-led} we show that a uniformly random proper colouring with sufficiently many colours satisfies the required conditions with non-zero probability.

The strategy for the second step is essentially the same as in Section~\ref{sec:general}; we still rely on a Rödl Nibble approach to recolour a fraction of the vertices in such a way that new witnesses are created for large-degree vertices. The fact that we are allowed to destroy some witnesses from the precolouring makes this second step much cheaper than the one from the previous section. This is presented in \cref{lem:pcf-precoloured2}.

\begin{lemma}
    \label{partial-pcf-led}
    Let $h\ge 20\log \Delta$ be an integer, and $G$ be a graph of maximum degree $\Delta$ and minimum degree $\delta \ge 16h$. Then, for every $d$ with $\delta \le d\le \Delta$, $G$ admits a proper $\threshold{V_{\le d}}{2h}{0}$-conflict-free $(\Delta+3d)$-colouring.
\end{lemma}

\begin{proof}
    Let $k \coloneqq \Delta+3d$, and let $\C_k(G)$ be the set of all proper $k$-colourings of $G$.
    Given a proper colouring $\sigma \in \C_k(G)$ and a vertex $v\in V(G)$, we let $L_\sigma(v) \coloneqq [k] \setminus \sigma(N(v))$ be the set of colours $x$ such that, if we redefine $\sigma(v) \gets x$, we still have $\sigma \in \C_k(G)$. We also denote $\sigma \setminus X$ the restriction of $\sigma$ to $V(G) \setminus X$, for every $X\subseteq V(G)$. 
    
    We fix some vertex $v\in V_{\le d}$, and we write $N(v)\coloneqq \{u_1, u_2, \ldots, u_{\deg(v)}\}$.
    Let $\bsig \in \C_k(G)$ be drawn uniformly at random; we want to analyse the probability of the bad event $B_{\bsig}(v)$ that $v$ has fewer than $2h$ solitary colours in $\bsig$. 
    To do that, we define the sequence of random colourings $\bsig_0, \ldots, \bsig_{\deg(v)} \in \C_k(G)$ where $\bsig_0 \coloneqq \bsig$, and for each $i\ge 1$, $\bsig_i$ is obtained from $\bsig_{i-1}$ be resampling $\bsig_{i-1}(u_i)$ uniformly at random from $L_{\bsig_{i-1}}(u_i)$. 
    This does not change the probability distribution of the random colourings; each $\bsig_i$ is uniformly distributed in $\C_k(G)$. 

    For every integer $0\le i\le \deg(v)$, let $\bS_i$ be the number of colours that appear exactly once on $\{u_1,\ldots,u_i\}$ in $\bsig_i$. Hence $\bS_{\deg(v)}$ is the number of witnesses of $v$ in $\bsig_{\deg(v)}$.
    For every $i\in [\deg(v)]$, let $\bX_i \coloneqq \bS_i - \bS_{i-1}$. Clearly we have that $\bX_i=-1$ if $\bsig_i(u_i)$ is one of the colours counted by $\bS_{i-1}$; $\bX_i=1$ if $\bsig_i(u_i) \notin \bsig_{i-1}(\{u_1, \ldots, u_{i-1}\})$; and $\bX_i=0$ otherwise. 
    Let $\sigma_{i-1}$ be any possible realisation of $\bsig_{i-1}$. Since there are always at least $3d$ colours available when resampling the colour of $u_i$, and at most $d$ colours that appear in $N(v)$, we have
    \begin{align*}
        \pr{\bX_i=1 \mid \bsig_{i-1}=\sigma_{i-1}} &\ge 2/3, \mbox{ and}\\
        \pr{\bX_i=-1 \mid \bsig_{i-1}=\sigma_{i-1}} &\le 1/3.
    \end{align*}

    We now condition on $\bsig = \sigma_0$, and prove that under that condition $\bS_{\deg(v)} \ge 2h$ with large enough probability. Note that under that condition, the realisation of $\bsig_{\deg(v)} \setminus N(v)$ is precisely $\sigma_0 \setminus N(v)$.
    Let $\{\bY_i\}$ be a sequence of Bernoulli random variables such that
    \[\bY_i=\begin{cases}
        1 &\mbox{if}~\bX_i=1,\\
        0 &\mbox{otherwise}.
    \end{cases}\]
    Clearly for every $i\in [\deg(v)]$ and every possible realisation $(y_1,\ldots,y_{i-1}) \in \{0,1\}^{i-1}$ of $\{\bY_i\}$, we have
    \[\pr{\bY_i = 1 \mid \bY_1 = y_1 , \ldots , \bY_{i-1} = y_{i-1}} \ge 2/3.\]
    Hence by \cref{lem:cond-binom} we have $\pr{\sum \bY_i \le t} \le \pr{B(\deg(v) , 2/3) \le t}$. If $B_{\bsig}(v)$ occurs, then it must hold that $\sum \bY_i < \deg(v)/2 + h$, and since $h \le \delta/16$, then in particular it must also hold that $\sum \bY_i < \frac{9}{16}\deg(v) = (1-5/32)\frac{2}{3}\deg(v)$. With an application of Chernoff's bound, we obtain
    \begin{align*}
        \pr{\bS_{\deg(v)} < 2h \mid \bsig=\sigma_0} &\le \pr{\sum \bY_i < \deg(v)/2 + h}\\
        &\le \pr{B(\deg(v),2/3) < (1-5/32) \frac{2\deg(v)}{3}} \\
        &\le \exp\left(-\frac{\deg(v)}{123}\right) \le \Delta^{-2.6} & \mbox{by \cref{lem:chernoff}\ref{it:upper}}.
    \end{align*}
    
    For a vertex $v\in V_{\le d}$, the outcome of $B_{\bsig}(v)$ is entirely determined by the colours assigned to vertices in $N(v)$. So if we fix the realisation of $\bsig$ outside of $N(v)$, we in particular fix the outcomes of all events $B_\bsig(u)$ such that $N(v) \cap N[u] = \emptyset$.
    Hence we let $\Gamma(v) \coloneqq N[N(v)]\setminus \{v\}$, of size $|\Gamma(v)| \le d\Delta$. 
    We let $\Sigma_0$ be the set of possible realisations of $\restrict{\bsig}{V(G) \setminus N(v)}$.
    With that definition of $\Gamma(v)$, for every $Z \subset V_{\le d} \setminus \Gamma(v)$, we infer that
    
    \[\pr{B_\bsig(v) \midbar \bigcap_{u \in Z} \overline{B_\bsig(u)}} \le \sup_{\sigma_0 \in \Sigma_0} \pr{B_{\bsig}(v) \midbar \restrict{\bsig}{V(G) \setminus N(v)} = \sigma_0} \le \Delta^{-2.6} \eqqcolon p.\]
    Since $\Delta \ge \delta \ge 16h \ge 320 \log \Delta$, we must have $\Delta > 2500$. So we have $4p(d\Delta) \le 4 \Delta^{-0.6} < 1$. We may therefore apply \cref{lem:llll} to conclude that, with non-zero probability, no event $B_{\bsig}(v)$ occurs.
\end{proof}

A direct application of \Cref{partial-pcf-led} with $d\coloneqq \Delta$ yields the following result as a corollary.

\begin{cor}
For every integer $h\ge 40\log \Delta$ and every graph $G$ of maximum degree $\Delta$ and minimum degree $\delta \ge 8h$, one has
\[ \pcf^h(G) \le 4\Delta.\]
\end{cor}
    
There remains to show how to transform a $\threshold{V_{\le d}}{2h}{0}$-conflict-free colouring into a conflict-free colouring through a random recolouring process. The main idea follows from \cref{lem:pcf-precoloured} with a better adaptation to the present situation.

\begin{lemma}\label{lem:pcf-precoloured2}
    Let $G$ be a graph of maximum degree $\Delta \ge 20000$, minimum degree $\delta \ge h_0 \ge 20\log\Delta$, and let $d \ge \delta$. Suppose there exists a proper $k$-colouring $\sigma_0$ of $G$ which is $\threshold{V_{\le d}}{h_0}{0}$-conflict-free. Let $\eta$ be an integer such that that $20 \log\Delta \le \eta \le \frac{d}{100}$, and let $m \coloneqq \ceil{8e\Delta\add/d}$. 
    Then there exists a proper $(k+m)$-colouring $\sigma$ of $G$ which is $\threshold{V_{\le d}}{h_0/2}{\add}$-conflict-free.
\end{lemma}

\begin{proof}
For $v \in V_{\le d}$, let $\W(v)$ be an arbitrary subset of $h_0$ witnesses of $v$ in $\sigma_0$ (here we use the fact that $\deg(v) \ge \delta \ge h_0$). We define $p \coloneqq \frac{1}{2\Delta}$. Since $\add \le d/100$, we easily verify that 
\begin{equation}\label{eq:mp2}
    2emp < 0.6 < e^{-1/2}.
\end{equation}

We execute the following procedure, which is identical to the one in the proof of \cref{lem:pcf-precoloured} except for the fact that we no longer ensure that witnesses in $\sigma_0$ remain witnesses in the final colouring.

\setlength{\interspacetitleruled}{0pt}%
\setlength{\algotitleheightrule}{0pt}%
\begin{algorithm}[H]
\textbf{Initialisation: }\For{$i \gets 1$ \KwTo $m$}{
    $A_i \gets \emptyset$\;
    \For{$u \in V(G)$}{
        Add $u$ to $A_i$ independently with probability $p$.\;
    }
}
$i_0\gets 1$\;
\For{$i \in [m]$}{
    $S_i \gets \sst{u \in V(G)}{N(u) \cap A_i \ne \emptyset}\ \cup\ \bigcup_{j = i_0}^{i-1} C_j$;

    $C_i \gets A_i \setminus S_i$.\;
    
    \textbf{fail-safe:} \If{there exists $v \in V_{>d}$ such that $|N(v) \cap \bigcup_{j = i_0}^{i} C_j| > \frac{\deg(v)}{2}$}{
        $i_0 \gets i+1$.\;
    }
}
\end{algorithm}\SetNlSty{texttt}{(}{)}

For $i \in [m]$, we respectively let $\bA_i$, $\bS_i$, $\bC_i$, and $\bi$ be the random values of the variables $A_i$, $S_i$, $C_i$, and $i_0$ at the end of the procedure. Let $\bsig$ be the colouring obtained from $\sigma_0$ after recolouring $\bC_i$ with the colour $k + i$ for each $i \in \{\bi, \dots, m\}$. As explained in the proof of \cref{lem:pcf-precoloured}, the colouring $\bsig$ is deterministically proper. For $v \in V_{>d}$, let $B_v$ be the (bad) event that more than $\deg(v)/2$ neighbours of $v$ belong in $\bA_1 \cup \dots \cup \bA_m$.

\begin{claim}\label{claim:reset2}
    For $v \in V_{>d}$, $\pr{B_v} \le \Delta^{-5}.$
\end{claim}

\begin{proof}
    Let $\bZ_v \coloneqq \sum_{i=1}^{m}|\bA_i \cap N(v)| \sim B(m \deg(v), p)$. By \cref{lem:binomial-tail-simple}, we obtain
    
    \[\pr{B_v} \le \pr{\bZ_v \ge \frac{\deg(v)}{2}}  \;\le\; (2emp)^{\deg(v)/2} < \Delta^{-5},\]
    where we have used \cref{eq:mp2} and the assumption $\deg(v) \ge \delta \ge 20\log\Delta$.
\end{proof} 

While we do not ensure deterministically that vertices of $\W(v)$ remain witnesses of $v$ in $\bsig$, for $v\in V_{\le d}$ we do expect that a majority of them will be.

\begin{fact}\label{fact:leftover}
    Let $v \in V_{\le d}$. Then $\W(v) \setminus (\bA_1 \cup \dots \cup \bA_m)$ is a set of witnesses of $v$ in $\sigma$.
\end{fact}

For $v\in V_{\le d}$, let $E_v^-$ be the (bad) event that more than $h_0/2$ vertices of $\W(v)$ belong in $\bA_1 \cup \dots \cup \bA_m$. Hence, if $E_v^-$ does not occur, $v$ still has at least $h_0/2$ witnesses in $\bsig$.

\begin{claim}\label{claim:leftover-witnesses}
    Let $v \in V_{\le d}$. Then $\pr{E_v^-} \le \Delta^{-5}$.
\end{claim}
\begin{proof}
Let $\bY_v \coloneqq \sum_{i=1}^{m} |\W(v) \cap A_i| \sim B(mh_0, p)$. By \cref{lem:binomial-tail-simple}, we obtain

\[\pr{E_v^-} \le \pr{\bY_v \ge h_0/2} \le \pth{2emp}^{h_0/2} < \Delta^{-5},\]
where we have used \cref{eq:mp2} and the assumption $h_0 \ge 20\log\Delta$.
\end{proof}

For $v \in V_{>d}$, let $\bX_v^i$ be the indicator variable for the event ``$|N(v) \cap \bC_i| = 1$''. The following fact is easily deduced from the disjointness of the sets $\bC_{\bi}, \dots, \bC_m$.

\begin{fact}\label{fact:keep-plus2}
    Let $v \in V_{>d}$. At the end of the procedure, $v$ has at least $\displaystyle \sum_{i=\bi}^m \bX_v^i$ witnesses in $\bsig$.
\end{fact}

For every $v\in V_{>d}$, let $\bX_v \coloneqq \sum_{i=1}^{m} \bX_v^i$. If $\bi = 1$ and $\bX_v \ge \add$, then $v$ has at least $\add$ witnesses in $\bsig$, as desired. Let $E_v$ be the (bad) event that $\bX_v < \add$.

\begin{claim}\label{claim:witnesses2}
    Let $v \in V_{>d}$. Then $\pr{E_v} \le \Delta^{-5}$.
\end{claim}

\begin{proof}
    Consider $i \in [m]$. Let us fix the realisation $\bA_1,\dots,\bA_{i-1}$ and the resulting value of $\bC_1,\dots,\bC_{i-1}$. Let $C'$ be the resulting value of the term ``$\bigcup_{j = i_0}^{i-1} C_j$'' at the start of round $i$. The distribution of $\bA_i$ is unaffected by this conditioning, i.e. every vertex of $v$ is still present in $\bA_i$ independently with probability $p$. Let $U \coloneqq N(v) \setminus C'$ and let $t \coloneqq |U|$. Because of the \textbf{fail-safe}, $C'$ contains at most $\deg(v)/2$ neighbours of $v$, therefore $t \ge \deg(v) - \frac{\deg(v)}{2} \ge \frac{d}{2}$.
    
    The probability that $N(v) \cap \bA_i = \{u\}$ with $u \in U$ is equal to $t p (1-p)^{\deg(v)-1} \ge \frac{dp}{2}(1-p)^{\Delta-1}$. Suppose such a vertex $u \in U$ exists, i.e. we are conditioning on the event ``$N(v) \cap \bA_i = \{u\}$ for some $u \in U$''. By independence of the random choices, the distribution of $\bA_i$ outside of $N(v)$ remains unaffected, i.e. every vertex $x \in V \setminus N(v)$ is present in $\bA_i$ independently with probability $p$. We now determine a lower bound on the probability that $u \in \bC_i$, which is equivalent to $u \notin \bS_i$ since $u \in \bA_i$. By definition of $U$, $u\notin C'$, so $u \notin \bS_i$ if and only $N(u)\cap \bA_i = \emptyset$. We already know that $N(v)\cap \bA_i = \{u\}$, therefore this happens with probability $(1-p)^{|N(u)\setminus N(v)|} \ge (1-p)^{\Delta}$. Therefore,

    \begin{equation}\label{eq:round-i-2}
        \pr{\bX_v^i = 1} \ge \frac{dp}{2}(1-p)^{2\Delta-1} > \frac{dp}{2e},
    \end{equation}
    where we have used the fact that $(1-\frac{1}{x})^{x-1} > \frac{1}{e}$ for all $x > 1$. Let $q \coloneqq \frac{dp}{2e}$. We have shown that $\pr{\bX_v^i = 1} > q$ regardless of the realisation of $\bA_1,\dots,\bA_{i-1}$, and so in particular, 

    \[\pr{\bX_v^i = 1 \midbar \bX_v^1,\dots,\bX_v^{i-1}} > q.\]
    
    By \cref{lem:cond-binom}, we obtain $\pr{\sum \bX_v^i < \add} \le \pr{B(m,q) < \add}$. We have defined $m$ such that $mq \ge 2\eta$, and since $\eta \ge 20\log\Delta$ we have $\frac{mq}{8} \ge 5\log\Delta$; by \cref{lem:chernoff2}\ref{it:left}, we conclude that $\pr{E_v} \le \pr{B(m,q) < \add} \le \Delta^{-5}$.
\end{proof}

The bad events $B_v$ and $E_v$ are associated to vertices of $V_{>d}$, while the bad events $E^-_v$ are associated to vertices of $V_{\le d}$. Each of these events
are determined by the realisation of $(\bA_1,\ldots,\bA_m)$ within the distance $2$ neighbourhood of their corresponding vertices, therefore these bad events are independent from all but at most $2\Delta^4$ other such bad events. By \cref{lem:lll}, with nonzero probability, none of these bad events occur. In that case, $\bi = 1$ and $\bX_v \ge \add$ for every $v \in V_{> d}$, while $|\W(v)\setminus (\bA_1\cup\dots\cup\bA_m)| \ge h_0/2$ for every $v \in V_{\le d}$. By \cref{fact:leftover} and \cref{fact:keep-plus2}, we conclude that the colouring $\bsig$ is $\threshold{V_{\le d}}{ h_0/2}{\add}$-conflict-free.
\end{proof}

\begin{cor}
    \label{cor:h-pcf2}
    Let $G$ be a graph of maximum degree $\Delta \ge 20000$ and minimum degree $\delta \ge 2000\log\Delta$. Let $h$ be an integer such that $20\log\Delta \le h \le \delta/100$. If $\delta \le\sqrt{h\Delta}$, then
    \[\pcf^h(G)\le\Delta + 17\sqrt{h\Delta}.\]
\end{cor}
\begin{proof}
    Let $d \coloneqq \ceil{3\sqrt{h\Delta}}$. By \cref{partial-pcf-led}, there exists a $(\Delta + 3d)$-colouring $\sigma_0$ that is proper and $\threshold{V_{\le d}}{2h}{0}$-conflict-free. We apply \cref{lem:pcf-precoloured2} with $k\coloneqq \Delta + 3d$, $h_0 \coloneqq 2h$ and $\add \coloneqq h$; the condition $\add \le d/100$ is verified by the assumptions $h \le \delta/100$ and $\delta \le d$. Then there exists an $h$-pcf $k$-colouring of $G$, with $k=\Delta + 3d + \ceil{8e\frac{\Delta}{d}h} \le \Delta + 17\sqrt{h\Delta}$.
\end{proof}

\begin{rk}
    In the statement of \cref{cor:h-pcf2}, we did not consider the case $\delta > \sqrt{h\Delta}$. In that case, \cref{thm:kamyczura} provides the upper bound $\pcf^h(G) \le \Delta + 30\frac{h\Delta}{\delta}$; observe that the second-order term is $\bigo{\sqrt{h\Delta}}$, and approaches $\bigo{h}$ as $\delta \rightarrow \Delta$.
\end{rk}

 \section{Discussion} 
\subsection{Tightness of the bounds}
As discussed in the introduction, the bound $\Delta + \bigo{\sqrt{h\Delta}}$ is unlikely to be tight. When $h=1$, we have a bound of the form $\Delta + O(\log \Delta)$, while when $h=2$ we have no better than $\Delta + \bigo{\sqrt{ \Delta \log\Delta}}$ with minimum degree $\Omega(\log \Delta)$. We expect the transition to be smoother between $h=1$ and larger values for $h$, as is suggested by \cref{conj:h-pcf}. At the other end of the spectrum, when $h$ gets close to $\delta$, it has been established in \cite{dai2023new} that a linear upper bound in terms of $\Delta$ is no longer possible, at least for regular graphs. This was formulated for $\odd^h(G)$, which we recall is a lower bound on $\pcf^h(G)$.

\begin{prop}[Dai, Ouyang, Pirot, 2024]
    For every even integer $\Delta \ge 2$ and $1\le t \Delta$, there is a $\Delta$-regular graph $G$ such that, letting $h\coloneqq \Delta+1-t$, one has
    \[ \odd^h(G) > \frac{1}{2}\frac{\Delta^2}{t+1}.\]
\end{prop}

The authors have also established a matching upper bound for $\odd^h(G)$ in that regime.

\begin{thm}[Dai, Ouyang, Pirot, 2024]
    Let $G$ be a $\Delta$-regular graph, let $\ceil{2(\ln \Delta + \ln \ln \Delta + 3)} \le t \le \Delta$ be a given integer, and let $h \coloneqq \Delta+1-t$. Then
    \[ \odd^h(G) = \bigO{\frac{\Delta^2}{t}}.\]
\end{thm}

One could wonder whether a similar upper bound holds for $\pcf^h(G)$ in that same regime. This is indeed the case, as it turns out that, when $h$ is sufficiently close to $\Delta$, then $\odd^h(G)$ and $\pcf^h(G)$ have a similar behaviour. This follows from the following observation.

\begin{prop}
    For every graph $G$ of maximum degree $\Delta$, and every integer $t < \frac{2\Delta}{3}$,
    \[ \pcf^{\Delta - 3t/2}(G) \le \odd^{\Delta-t}(G).\]
\end{prop}

\begin{proof}
    Let $h \coloneqq \Delta-t$, and let $\sigma$ be any $h$-odd colouring of $G$. 
    Let $v\in V(G)$; we say that $c\in \sigma(N(v))$ is an \emph{odd colour of $v$ in $\sigma$} if it has an odd number of occurrences in $N(v)$. 
    Let $x\ge 0$ be the number of odd colours of $v$ that have $> 1$ occurrences in $N(v)$. There can be at most $\Delta-3x$ other (odd) colours in $N(v)$, hence $h\le \Delta - 2x$. 
    We conclude that $x \le \frac{\Delta-h}{2} = t/2$, and so $v$ has at least $\Delta - 3t/2$ solitary colours in $\sigma$. This holds for every $v\in V(G)$, so $\sigma$ is $(\Delta - 3t/2)$-conflict-free.
\end{proof}

\subsection{Algorithmic aspects}

We note that the proof of \Cref{lem:partial-pcf} is constructive with the description of a greedy algorithm that returns the desired solution. 
While the statements of \Cref{lem:pcf-precoloured} and \Cref{lem:pcf-precoloured2} are purely existential, the only non-constructive part in their respective proofs is the application of \Cref{lem:lll}. If one replaces the latter with the following algorithmic version of the Lov\'asz Local Lemma due to Moser and Tardos~\cite{MoTa10}, then both these results can be obtained constructively with a randomised algorithm of expected polynomial time.

\begin{lemma}[Algorithmic Lov\'{a}sz Local Lemma]\label{lem:alll}
Let $\mathcal{B} = \{B_1, \ldots, B_n\}$ be a finite set of random (bad) events that are all determined by a finite set of random independent variables $\mathcal{X}$, and all of which happen with probability at most $p$. Suppose that, for every $i\in [n]$, $B_i$ is independent from all but at most $d$ other events.

If $epd\le 1$, then there exists a valuation of the variables in $\mathcal{X}$ such that no bad event in $\mathcal{B}$ happens, and such a valuation can be constructed with a randomised algorithm of expected runtime polynomial in $n$ and $|\mathcal{X}|$.
\end{lemma}

Indeed, in the proofs of \Cref{lem:pcf-precoloured} and \Cref{lem:pcf-precoloured2}, the bad events are all determined by the set of independent random variables $\mathcal{X} \coloneqq \{ \mathds{1}_{v \in A_i} : i \in [m], v\in V(G)\}$.
We conclude that one can obtain the colouring promised by \Cref{thm:pcf-upper} in expected polynomial time. 

In contrast, the proof of \Cref{partial-pcf-led} relies on the lopsided version of LLL (\Cref{lem:llll}) which seems to be much harder to prove constructively (see \cite{HaVo20} for a partial result in that direction), and more crucially on a much more complex random object, namely a uniformly random proper colouring $\bsig$ of the graph $G$. Sampling $\bsig$ --- or a good enough approximation thereof --- in polynomial time is by itself a notoriously hard open problem. If one wishes to make that proof constructive, the most promising strategy in our opinion would be to apply directly the Entropy Compression Method used by Moser and Tardos in their proof of \Cref{lem:alll} on the following random algorithm, which we conjecture should return the desired colouring in expected polynomial time.

\medskip
\setlength{\interspacetitleruled}{0pt}%
\setlength{\algotitleheightrule}{0pt}%
\begin{algorithm}[H]
Fix an arbitrary ordering on $V(G)$\;
$k \gets \Delta + 3d$\;
$\sigma \gets$ any proper $k$-colouring of $G$ (constructed greedily)\;
\While{there is a vertex $v$ of degree at most $d$ with less than $2h$ witness in $\sigma$}{
    \ForEach{colour class $C$ in $N(v)$}{
        \ForEach{$u \in C \setminus \min C$}{
             $\sigma(u) \gets$ uniformly random colour from $[k] \setminus \sigma(N(u))$.
        }
    }
}
\Return{$\sigma$}
\end{algorithm}\SetNlSty{texttt}{(}{)}

\medskip
Since the algorithmic aspects of our proofs are not the focus of the paper, we leave that for future work.

\bibliographystyle{abbrv}
\bibliography{reference}

\end{document}